\documentclass{amsproc}
\usepackage{mystyle}

\newtheorem{lem}[thm]{Lemma}
\theoremstyle{definition}
\newtheorem{defn}[thm]{Definition}

\declaretheorem[name=Example,qed={\lower-0.3ex\hbox{$\triangleleft$}}, sibling=thm]{example}

\usepackage[backend=biber,style=alphabetic, citestyle=alphabetic, sorting = nty, maxnames=50]{biblatex}
\title{Group-graded Twisted Calabi--Yau Algebras}
\author{Yasmeen S. Baki}
\date{}

\begin{document}
\begin{abstract}
Historically, the study of graded (twisted or otherwise) Calabi--Yau algebras has meant the study of such algebras under an $\N$-grading. In this paper, we propose a suitable definition for a twisted $G$-graded Calabi--Yau algebra, for $G$ an arbitrary abelian group. Building on the work of Reyes and Rogalski, we show that a $G$-graded algebra is twisted Calabi--Yau if and only if it is $G$-graded twisted Calabi--Yau. In the second half of the paper, we prove that localizations of twisted Calabi--Yau algebras at elements which form both left and right denominator sets remain twisted Calabi--Yau. As such, we obtain a large class of $\Z$-graded twisted Calabi--Yau algebras arising as localizations of Artin--Schelter regular algebras. Throughout the paper, we survey a number of concrete examples of $G$-graded twisted Calabi--Yau algebras, including the Weyl algebras, families of generalized Weyl algebras, and universal enveloping algebras of finite dimensional Lie algebras.
\end{abstract}

\maketitle

\section{Introduction}
\label{sec:intro}

Throughout this paper, fix a field $k$ and assume that all bimodules are $k$-central. Unless otherwise specified, all modules are assumed to be left modules, and when we speak of a graded ring $R = \bigoplus_{\sigma \in G} R_{\sigma}$, we assume that $G$ is an abelian group. The undecorated tensor $\otimes$ is understood to be taken over the base field $k$.

First introduced in late 2006 by Ginzburg \cite{GinzburgCY} as a way to ``transplant conventional Calabi--Yau geometry'' into the non-commutative setting, Calabi--Yau algebras are now a mainstay of noncommutative algebra, and occupy a space in many a noncommutative geometer's consciousness that is difficult to overstate. Often thought of as a good candidate for noncommutative polynomial rings in the $\N$-graded case, the study of $\N$-graded twisted Calabi--Yau algebras encompasses that of Artin--Schelter regular algebras \cite{cytoolkit}, and in global dimension $3$ is closely related to path algebras of quivers arising from superpotentials \cite{BOCKLANDT200814}. Work of Reyes and Rogalski in \cite{cytoolkit} has classified locally finite-dimensional $\N$-graded twisted Calabi--Yau algebras in small dimensions, and recent work of Gaddis et al. in \cite{GADDIS202486} classifies certain quivers supporting $\N$-graded twisted Calabi--Yau algebras, extending work done by Gaddis and Rogalski in \cite{GADDIS2021106645}. Inspired by this interest in the $\N$-graded case, our initial motivation is to ask what other types of gradings are supported by these twisted Calabi--Yau algebras.

Historically, the study of graded (twisted or otherwise) Calabi--Yau algebras has meant the study of such algebras under an $\N$-grading. This makes good sense: if we are to understand this branch of noncommutative algebra as searching for suitable analogs to polynomial rings, then noncommutative rings with some notion of ``polynomial degree'' are the ideal place to start. Yet various sources allude to the possibility of group-graded Calabi--Yau algebras being an interesting direction of study. For instance, \cite{CY-deformations} defines explicitly the (untwisted) $\Z$-graded flavors, while Bell and Rogalski, in their classification of $\Z$-graded simple rings, motivate the discussion as a response to the focus of noncommutative algebraic geometry resting almost entirely in the $\N$-graded setting \cite{zgrad}. The lack of attention to group-graded twisted Calabi--Yau algebras happens in spite of the fact that examples such as the Weyl algebras are well-known to be simultaneously $\Z$-graded and Calabi--Yau. 

Our preliminary goal is thus to develop a proper place for the Weyl algebras in the theory of graded twisted Calabi--Yau algebras. The fact that the Weyl algebras are not non-trivially $\N$-graded is further evidence of our hope that there really is something new to be said in this group-graded setting. Our first order of business is then in developing the machinery to provide a suitable definition of precisely \textit{what} we mean by ``group-graded twisted Calabi--Yau''.

Let us begin by recalling some of the preliminary definitions. Let $A$ be a $k$-algebra, and define the enveloping algebra of $A$ as $A^e \coloneqq A \otimes A^{\op}$, where $A^{\op}$ is the opposite algebra of $A$. As we are assuming that all bimodules are $k$-central, we know that given an $(A,A)$-bimodule $M$, we can view it as a left or right $A^e$-module via the action 
\begin{align*}
    (a \otimes b^{\op}) \cdot m = a \cdot m \cdot b = m \cdot (b \otimes a^{\op})\text,
\end{align*}
where $a,b \in A$, and $m \in M$. In particular, we have a correspondence between left $A^{e}$-modules and $(A,A)$-bimodules. The algebra $A$ is said to be \textit{homologically smooth} if it has a finite-length resolution of finitely generated projective $A^{e}$-modules; that is, $A$ is perfect as a left $A^{e}$-module. If $\nu$ is an algebra automorphism of $A$, then we let $^1A^{\nu}$ denote the bimodule $A$ with normal left action, but right action twisted by $\nu$. That is, $a \cdot b \cdot c \coloneqq ab\nu(c)$ for all $a,b,c \in A$. Let us first recall the definition of a twisted Calabi--Yau algebra, before any grading enters the picture.

\begin{defn}[\cite{cytoolkit}]
    Let $A$ be a $k$-algebra, and suppose that $A$ is homologically smooth. Then we say $A$ is \emph{twisted Calabi--Yau of dimension $d$} if there is an invertible $(A,A)$-bimodule $U$ such that following holds as an isomorphism of $(A,A)$-bimodules:
    \begin{equation*}
        \ext^i_{A^e}(A,A^e) \cong 
        \begin{cases}
            0, \quad &i\neq d; \\
            U, \quad &i = d.
        \end{cases}
    \end{equation*}
\end{defn}
The use of the term \textit{twisted} is meant to emphasize that $U$ may not be isomorphic to $A$. If the above isomorphism should hold with $U \cong A$, then we say that $A$ is (untwisted) Calabi--Yau. If the isomorphism were to hold with $U \cong {}^1A^{\nu}$ for $\nu$ an algebra automorphism of $A$, then we say that $A$ is twisted Calabi--Yau with Nakayama automorphism $\nu$. We say that the graded $k$-algebra $A = \bigoplus_{\sigma \in \N} A_{\sigma}$ is $\N$-graded twisted Calabi--Yau of dimension $d$ if $A$ is perfect when viewed as an object in the category of $\N$-graded left $A^e$-modules, and the above isomorphism holds when viewed as a morphism in the same graded category. The more general $G$-graded situation is defined as follows.

\begin{defn}[Definition \ref{def:g-grade-CY}]
    Let $A$ be a $G$-graded $k$-algebra, and let $U$ be a $G$-graded invertible $(A,A)$-bimodule. We say that $A$ is \emph{$G$-graded twisted Calabi--Yau of dimension $d$} if $A$ is graded homologically smooth and we have the following graded isomorphism of $(A,A)$-bimodules:
\begin{align*}
    \Ext_{A^{e}}^{i}(A,A^e) \cong
    \begin{cases}
    0, \ i \neq d; \\
    U, \ i = d.
    \end{cases}
\end{align*}
\end{defn}

As we would hope for any good definition, we observe that the familiar $\N$-graded flavor is a special case of this new definition. Indeed, any $\N$-graded algebra $A = \bigoplus_{n \in \N} A_n$ can be realized as a $\Z$-graded algebra by setting $A_n = 0$ for $n < 0$. Reyes and Rogalski prove in Theorem 4.2 of \cite{cytoolkit} that an $\N$-graded algebra that happens to be twisted Calabi--Yau is the same as an $\N$-graded twisted Calabi--Yau algebra. We find the same holds true in the group-graded setting, as well.

\begin{thm}[Theorem \ref{thm:gradcyiff}]
    Let $A$ be a $G$-graded algebra. Then $A$ is twisted Calabi--Yau of dimension $d$ if and only if $A$ is $G$-graded twisted Calabi--Yau of dimension $d$.
\end{thm}

With this theorem in hand, we now have an easy way of locating examples of group-graded twisted Calabi--Yau algebras from known examples: any algebra that is simultaneously $G$-graded and twisted Calabi--Yau is G-graded twisted Calabi--Yau. We initially would like to survey some of the examples that can occur. In addition to the Weyl algebras, we present further examples of $\Z$-graded Calabi--Yau algebras including examples arising from generalized Weyl algebras, as well as universal enveloping algebras of finite dimensional Lie algebras.

In Section \ref{sec:localization} of this paper, we use localization at left and right denominator subsets to generate $G$-graded Calabi--Yau algebras from known $\N$-graded examples.

\begin{thm}[Theorem \ref{thm:localization}]
    Let $A$ be a $k$-algebra, and let $S$ be a left and right denominator subset of $A$. If $A$ is twisted Calabi--Yau of dimension $d$, then $S^{-1}A$ is twisted Calabi--Yau of dimension $d$.
\end{thm}

Perhaps a surprising consequence of this theorem is that the field $k(t_1, \dots, t_n)$ provides an example of a (trivially) $\Z$-graded Calabi--Yau algebra of dimension $n$. This signals a shift from the usual perspective of Calabi--Yau algebras being a suitable candidate for noncommutative polynomial spaces, and that perhaps something more is going on here. 

\subsection*{Acknowledgements} 
This work constitutes a portion of the author's PhD thesis. She is advised by Manuel Reyes, and thanks him for his feedback, encouragement, and proofreading of this paper. She also thanks the referee for their careful reading of the paper and providing suggestions which have improved its quality and clarity.

\section{Homologically smooth algebras and invertible bimodules: The group-graded flavors}
\label{sec:homsmooth}

The purpose of this section is to build a repertoire of preliminary results concerning the group-graded versions of well-known results regarding homologically smooth algebras and invertible bimodules in the $\N$-graded setting. Most of the proofs will follow almost immediately from existing arguments (which we point to in the literature), but with careful extensions to the more general group-graded case and detail-checking provided.

Let us first agree on a few conventions surrounding graded rings and their module categories. Let $G$ denote an abelian group with identity $e$, and whose operation we write multiplicatively. For a ring $R$ to be $G$-graded means that $R$ has a decomposition into additive subgroups $R = \bigoplus_{\sigma \in G} R_\sigma$ such that $R_{\tau}R_{\sigma} \subseteq R_{\tau\sigma}$ for all $\tau,\sigma \in G$. Unless otherwise specified, ``module'' by itself is understood to mean a left module. Let $R$ be a $G$-graded ring. For an $R$-module $M$ to be $G$-graded means that $M$ has a decomposition into additive subgroups $M = \bigoplus_{\sigma \in G} M_{\sigma}$ such that $R_{\tau} M_{\sigma} \subseteq M_{\tau \sigma}$ for all $\tau,\sigma \in G$. Given a $G$-graded ring $R$, by $\catname{_RMod}$ we mean the usual category of left $R$-modules, and by $\grmod{G}$ we mean the category of $G$-graded left $R$-modules. Objects in this graded category are $G$-graded left $R$-modules, and morphisms are the degree-preserving morphisms of $\catname{_RMod}$. When $G$ is clear from context, we will often write $\grmodr$. 

Let $\Hom_R(-,-)$ denote homomorphism groups of left modules, and let $\Hom_{R^{\op}}(-,-)$ denote those of right modules. Further, we require that homomorphisms of $R$-modules act opposite that of scalars. Given $M,N$ objects in $\grmodr$ and $f \in \Hom_R(M,N)$, we say that $f$ is a graded morphism of degree $\tau \in G$, if $f(M_{\sigma}) \subseteq N_{\sigma \tau}$ for all $\sigma \in G$, and we write $f \in \Hom_{R}(M,N)_{\tau}$. Analogously for the right, we say that $f \in \Hom_{R^{\op}}(M,N)$ is a graded morphism of degree $\tau \in G$, if $f(M_{\sigma}) \subseteq \tau \sigma$ for all $\sigma \in G$. We define the ``total'' Hom-group as follows:
\begin{align*}
    \HOM(M,N) \coloneqq \bigoplus_{\tau \in G} \Hom_R(M,N)_{\tau} \subseteq \Hom_R(M,N).
\end{align*}
We point out that $\HOM(M,N)_e = \Hom_{\grmodr}(M,N)$, and that in general, the inclusion above is strict. However, it is well understood that in the case where $M$ is finitely generated in $\grmodr$, equality holds \cite[Corollary 2.4.4]{sacred2}. We can similarly define the total right-derived functor of $\HOM_R(-,-)$, which we will denote as $\Ext^i_R(-,-)$ for $i \geq 0$.

\subsection{Homologically smooth algebras}

Define the enveloping algebra of a $k$-algebra $A$ as $A^{e} := A \otimes_k A^{\op}$ and recall that $A$ is said to be \textit{homologically smooth} if it has a finite length resolution of finitely generated projectives in the category of $A^e$-modules; that is, $A$ is perfect as a left $A^e$-module. The graded version of this definition is a straightforward extension: if $A$ is $G$-graded, then we consider such a resolution in the $G$-graded $A^e$-module category. In developing graded versions of definitions, a common theme becomes the question of when the graded and ungraded versions are equivalent. As an example of such nuances, consider the fact that a projective object in $\grmodr$ is the same as a graded and projective object in $\rmod$, but the same is not the case if we are to replace ``projective'' with ``injective'', or perhaps even more surprisingly ``free'' (cf. \cite[p.21]{sacred2}). However, the graded and ungraded definitions of homologically smooth interact nicely. The following result is an adaptation of \cite[Lemma 2.4]{cytoolkit}, where it was proven in the $\N$-graded case.
 
\begin{lem} \label{gradedhomsmooth}
    Let $A$ be a $G$-graded $k$-algebra. Then $A$ is homologically smooth if and only if it is $G$-graded homologically smooth.
\end{lem}
\begin{proof}
    Let $A$ a $G$-graded $k$-algebra. Suppose first that $A$ is $G$-graded homologically smooth. Since the forgetful functor from $\grmodchoice{A^e}$ to $\catname{_{A^e}Mod}$ which removes grading preserves projectives, it follows that $A$ is homologically smooth. Conversely, suppose that $A$ is homologically smooth with the following resolution of finitely generated and projective $A^e$-modules:
    \begin{align*}
        0 \to P_{n} \xrightarrow{} \dots \to P_{1} \xrightarrow{} P_{0} \xrightarrow{d_0} A \to 0
    \end{align*}
    Since $A$ is graded and finitely generated as an $A^e$-module, a finite set of homogeneous generators can be selected for $A$ (i.e., $A$ is graded finitely generated). In particular, we can select $Q_0$ a projective and finitely generated object of $\grmodchoice{A^e}$ so that we have the graded surjection $Q_0 \xtwoheadrightarrow{\varepsilon_0} A$. We now continue this selection of $Q_i$ inductively. That is, suppose we have found $Q_i$ with corresponding graded surjections $\varepsilon_{i}: Q_i \to \ker(\varepsilon_{i-1})$ for $0 \leq i < n - 1$, at which point we can consider the exact sequences
    \begin{align*}
        0 \to \ker(d_{i}) \to P_i \to P_{i-1} \to \dots \to P_1 \to P_0 \to A \to 0\text, \\
        0 \to \ker(\varepsilon_{i}) \to Q_i \to Q_{i-1} \to \dots \to Q_1 \to Q_0 \to A \to 0.
    \end{align*}
   Since $P_{i+1} \xtwoheadrightarrow{} \ker(d_{n-1})$, it follows that $\ker(d_{n-1})$ is finitely generated, and an application of (generalized) Schanuel's Lemma gives the isomorphism
    \begin{align*}
        \ker(d_i) \oplus Q_i \oplus P_{i-1} \oplus \cdots \cong \ker(\varepsilon_i) \oplus P_i \oplus Q_{i-1} \oplus \cdots
    \end{align*}
    holding in $\catname{_{A^e}Mod}$. Since the entirety of the left-hand side is finitely generated, it follows that $\ker(\varepsilon_i)$ must be finitely generated as well. In particular, $\ker(\varepsilon_i)$ is graded finitely generated, so we can select $Q_{i+1} \in \grmodchoice{A^e}$ graded finitely generated and projective with $\varepsilon_{i+1}: Q_{i+1} \xtwoheadrightarrow{} \ker(\varepsilon_{i})$. Continue this process until $Q_{n-1}$ has been constructed, as which point we compare the following two exact sequences:
    \begin{align*}
        0 \to P_{n} \xrightarrow{} \dots \to P_{1} \xrightarrow{} P_{0} \xrightarrow{} A \to 0, \\
        0 \to \ker(\varepsilon_{n-1}) \to Q_{n-1} \dots \to Q_1 \to Q_0 \to A \to 0.
    \end{align*} 
   Yet another application of (generalized) Schanuel's Lemma yields the following isomorphism of $A^e$-modules:
    \begin{align*}
        P_n \oplus Q_{n-1} \cdots \cong \ker(\varepsilon_{n-1}) \oplus P_{n-1} \oplus \cdots.
    \end{align*}
    Since the entirety of the left hand side is projective and finitely generated, the same holds true of $\ker(\varepsilon_{n-1})$. In particular, $\ker(\varepsilon_{n-1})$ is a finitely generated and projective object in $\catname{gr_{A^e}Mod}$, and we see that the constructed resolution guarantees that $A$ is $G$-graded homologically smooth, as desired.
\end{proof}

\subsection{Invertible bimodules}
\label{sec:invertbimod}

As was done for homologically smooth algebras, we now wish to extend results regarding invertible bimodules to their $G$-graded counterparts. Let $A$ a $G$-graded $k$-algebra. In the ungraded world, an $(A,A)$-bimodule $U$ is said to be invertible if there exists an $(A,A)$-bimodule $V$ such that $U \otimes_A V \cong V \otimes_A U \cong A$ as $(A,A)$-bimodules. To define the $G$-graded version, we need first make sense of the graded tensor product, which we denote by $\underline{\otimes}$. We say that the degree $\gamma$ component of $U \underline{\otimes}_A V$ is the vector subspace spanned by all pure tensors of the form $u_\tau \otimes v_\sigma$ such that $\tau\sigma=\gamma$. That is, $(U \underline{\otimes}_A V)_\gamma = \bigoplus_{\sigma \in G} U_{\gamma \sigma^{-1}} \otimes_A V_{\sigma}$. Now, if $U$ is an $(A,A)$-bimodule that carries a $G$-grading, then we say that $U$ is \emph{$G$-graded invertible} if there exists $V$ a $G$-graded $(A,A)$-bimodule such that $U \underline{\otimes}_A V \cong V \underline{\otimes}_A U \cong A$ as graded $(A,A)$-bimodules.

Of special importance to us is the graded and invertible $(A,A)$-bimodule $^1A^\nu$. This bimodule should be understood as the $(A,A)$-bimodule $A$ with normal left action, but right action twisted by $\nu$, a graded automorphism of degree $\tau$ of $A$. In particular, for all $a,b,c \in A$ we have the following:
$$
a \cdot b \cdot c := ab\nu(c).
$$
Recall further that we have made the assumption that all bimodules $M$ be $k$-central; that is, scalars from $k$ commute with elements of $M$. The reason for this restriction is twofold. Insisting on bimodules being $k$-central guarantees that given an $(A,A)$-bimodule, we can view it as a left or right $A^e$-module via the action 
\begin{align*}
    (a_1 \otimes a_2^{\op}) \cdot m = a_1 \cdot m \cdot a_2 = m \cdot (a_2 \otimes a_1^{\op})\text.
\end{align*}
Note that since our grading group $G$ is abelian, this also is enough to establish a correspondence between $G$-graded $A^e$-modules and $G$-graded $(A,A)$-bimodules. Additionally, $k$-centrality of $(A,A)$-bimodules will induce a $k$-linear Morita autoequivalence of $\mymod{A}$. This will be particularly important in our next lemma, which is an yet another update of \cite[Lemma 2.10]{cytoolkit}, where we prove the result holds in the more general $G$-graded case.

\begin{lem} \label{gradedinvertible}
    Let $A$ be a $G$-graded $k$-algebra, and let $U$ be a $G$-graded $(A,A)$-bimodule. Then $U$ is invertible if and only if it is $G$-graded invertible.
\end{lem}

\begin{proof}
    Let $A$ be a $G$-graded $k$-algebra, and let $U$ be a $G$-graded $(A,A)$-bimodule. Suppose that $U$ is invertible. Then, $U \otimes_A -: \mymod{A} \to \mymod{A}$ provides an autoequivalence of the category of left $A$ modules, and as such $U$ is finitely generated and projective on the left and right, as well as a generator for both $\mymod{A}$ and $\catname{Mod\text{-}A}$. Finite generation of $U$ as a left module gives equality of the total Hom-group $\HOM_A(U,A)$ and $\Hom_A(U,A)$ on the level of sets, and the map which sends a graded morphism to its ungraded counterpart provides an isomorphism of abelian groups. We similarly have $\HOM_{A^{\op}}(U,A) \cong \Hom_{A^{\op}}(U,A)$ as abelian groups. 

    Now, standard results of Morita theory (cf. \cite[Lemma 18.17]{sacred3}) assert that 
    \begin{align*}
        &\Phi: \Hom_{A^{\op}}(U,A) \otimes_A U \to A \\
        (\text{resp.}) &\Psi: U \otimes_A \Hom_{A}(U,A) \to A
    \end{align*}
    given by $f \otimes x \mapsto f(x)$ (resp. $x \otimes f \mapsto (x)f$) provides an isomorphism of $(A,A)$-bimodules, and it remains to show that $\Phi$ and $\Psi$ preserve grading. Towards that end, recall that the forgetful functor from the graded to ungraded module categories reflects isomorphisms, so it suffices to argue that $\Phi$ and $\Psi$ are morphisms in the category of graded left $A^{e}$-modules. 

    Let $\Hom_{A^{\op}}(U,A) \otimes_A U = \bigoplus_{\sigma \in G} M_{\sigma}$ where $M_{\sigma}$ is the additive subgroup generated by all elements of the form $f \otimes x$ where $f \in \Hom_{A^{\op}}(U,A)$ has degree $\tau$ and $x \in U$ has degree $\gamma$ such that $\tau \gamma = \sigma$. Consider $f$ and $x$ with these degrees, and that $\Phi(f \otimes x) = f(x) \in A_{\tau \gamma = \sigma}$, so in particular $\Phi(M_{\sigma}) \subseteq A_{\sigma}$ for all $\sigma \in G$. Hence, $\Phi$ provides the following graded isomorphism of $A^e$-modules:
    \begin{align*}
        \Phi: \HOM_{A^{\op}}(U,A) \underline{\otimes}_A U \xrightarrow{\sim} A.
    \end{align*}
    We can similarly define a grading on $U \otimes_A \Hom_A(U,A)$ and check that $\Psi$ preserves grading, taking care to remember that the graded morphism $f \in \Hom_A(U,A)$ acts opposite that of scalars. By a symmetric argument, we end up with the graded isomorphism 
    \begin{align*}
        \Psi: U \underline{\otimes}_A \HOM_{A}(U,A) \xrightarrow{\sim} A\text.
    \end{align*}
    It follows that $U$ is a $G$-graded invertible $(A,A)$-bimodule with inverse bimodule given by $V \cong \HOM_{A^{\op}}(U,A) \cong \HOM_A(U,A)$.
    
    The converse is readily seen by forgetting the grading on $U \underline{\otimes}_A V \cong V \underline{\otimes}_A U \cong A$, and observing that we still have isomorphisms of $(A,A)$-bimodules. We conclude that $U$ is invertible if and only if it is $G$-graded invertible, as desired.
\end{proof}

\section{Group-graded twisted Calabi--Yau algebras} 

\subsection{Definition and Compatibility with Bimodule Structure}
\label{sec:CYdef}
We begin our study by first proposing a definition for group-graded twisted Calabi--Yau algebras which closely mimics the definition of the familiar $\N$-graded case, as given in \cite[Definition 4.1]{cytoolkit}, for example. As would hold for any good definition, we will see that gradings by monoids which embed into groups can be recovered by setting certain pieces to the trivial vector space. Indeed, we will realize the $\N$-graded case as just a $\Z$-grading with negative degree components equal to zero.

\begin{defn} \label{def:g-grade-CY}
Let $A$ be a $G$-graded $k$-algebra. We say that $A$ is \emph{$G$-graded twisted Calabi--Yau of dimension $d$} if $A$ is graded homologically smooth and there exists a $G$-graded invertible $(A,A)$-bimodule $U$ such that the following holds as a graded isomorphism of $(A,A)$-bimodules:
\begin{align*}
    \Ext_{A^{e}}^{i}(A,A^e) \cong
    \begin{cases}
    0, \ i \neq d; \\
    U, \ i = d.
    \end{cases}
\end{align*}
\end{defn}

A few comments are in order. Consider $^1A^\nu$, the graded $(A,A)$-bimodule $A$, with normal left action, but right action twisted by $\nu$, a graded algebra automorphism of $A$. Should Definition $\ref{def:g-grade-CY}$ hold with $U \cong {^1}A^{\nu}(\sigma)$ for some $\sigma \in G$, then we say that $\nu$ is a graded Nakayama automorphism of $A$ and is of Gorenstein index $\sigma$. Should it further hold with $\nu$ the identity map, then we say that $A$ is simply $G$-graded (untwisted) Calabi--Yau. 

With this definition in hand, we may be tempted to ask \textit{how different} are twisted Calabi--Yau algebras that happen to be $G$-graded from $G$-graded twisted Calabi--Yau algebras. Reyes and Rogalski in \cite[Theorem 4.2]{cytoolkit} shows that in the $\N$-graded case, they happen to be equivalent. With a slight modification of their proof, we will see that the same holds true for any group grading. 

\begin{thm} \label{thm:gradcyiff}
    Let $A$ be a $G$-graded algebra. Then $A$ is twisted Calabi--Yau of dimension $d$ if and only if $A$ is $G$-graded twisted Calabi--Yau of dimension $d$.
\end{thm}

\begin{proof}
    Observe first that Lemma \ref{gradedhomsmooth} gives that $A$ is $G$-graded and homologically smooth if and only if $A$ is $G$-graded homologically smooth. Whether $A$ is homologically smooth or graded homologically smooth, we know that $\HOM_{A^e}(P_i,A^e) \cong \Hom_{A^e}(P_i,A^{e})$ for $P_i$ a finitely generated and projective $A^e$-module, so in taking cohomology we have $\Ext^i_{A^e}(A,A^e) \cong \ext^i_{A^e}(A,A^{e})$ for all $i \geq 0$. Further, in these cases $\Ext^{d}_{A^e}(A,A^e) \cong \ext^d_{A^e}(A,A^e)$ is invertible if and only if it is graded invertible by Lemma \ref{gradedinvertible}. We conclude that a $G$-graded algebra $A$ is twisted Calabi--Yau of dimension $d$ if and only if it is $G$-graded twisted Calabi--Yau of dimension $d$.
\end{proof}

\begin{example} \label{ex:weylcy} [Weyl Algebras with Standard $\Z$-grading, $\Z^n$-grading]
Let $k$ be a field of characteristic zero. We consider the $n$\textsuperscript{th} Weyl algebra, $A_n$, defined as follows:

\begin{align*}
    A_n := k\langle x_1, \dots, x_n, y_1, \dots, y_n \rangle / ([x_i,y_j] - \delta_{ij}, [x_i,x_j],[y_i,y_j])\text,
\end{align*}
where $[-,-]$ denotes the commutator. We observe that these Weyl algebras possess a $\Z$-grading by setting $\deg(x_i) = 1$ and $\deg(y_i) = -1$ for $1 \leq i \leq n$. On the other hand, we could instead consider $V$ a vector space with a symplectic bilinear form $(-,-)$ and define the Weyl algebra over $V$ by $TV/(xy - yx - (x,y))$. If we let $V$ have basis $(x_1, \dots, x_n, y_1, \dots y_n)$, and form $(x_i,y_j) = \delta_{ij} = -(y_j,x_i), (x_i,x_j) = 0 = (y_i,y_j)$ for $1 \leq i,j \leq n$, then these two notions of Weyl algebras coincide (cf. \cite[Exercise 1.2.3]{deformations}).

For this example, we take the symplectic vector space view. It is well-known that $A_n$ provides an example of a Calabi--Yau algebra of dimension $2n$, and projective resolutions can be found in references such as \cite{suarezalvarez2001algebra} and \cite{deformations}. Since $A_n$ carries a $\Z$-grading, Theorem \ref{thm:gradcyiff} immediately tells us that $A_n$ is $\Z$-graded twisted Calabi--Yau of dimension $2n$. Let us now show that $A_n$ is $\Z$-graded Calabi--Yau with trivial Gorenstein index.

We first consider the resolution of $A_n$ given by \cite{deformations}:
\[
    0 \to A_n \otimes \bigwedge^{2n} V \otimes A_n \to A_n \otimes \bigwedge^{2n-1}V \otimes A_n \to \cdots \to A_n \otimes A_n \twoheadrightarrow{} A_n \to 0,
\]
\begin{align*}
    f \otimes (v_1 \wedge \cdots \wedge v_i) \otimes g \mapsto \sum_{j = 1}^{i}&(-1)^{j-1}(fv_j) \otimes (v_1 \wedge \cdots \hat{v}_j \cdots \wedge v_i) \otimes g \\
    +&(-1)^jf \otimes (v_1 \wedge \cdots \hat{v}_j \cdots \wedge v_i) \otimes (v_jg).
\end{align*}
Our goal is to show that this provides a $\Z$-graded resolution of $A_n$. Recall that if $V$ is a $\Z$-graded vector space, then we have an induced $\Z$-grading on the tensor algebra $T(V)$ as follows: given $v = v_1 \otimes \dots \otimes v_n \in T^n(V)$, we can define $|v| := \sum_{i = 1}^{n} |v_i|$ (cf. \cite[p. 58]{Keller2019}). From here, we have an induced $\Z$-grading on the exterior algebra of $V$, given by considering $T(V) = \bigoplus_{m \in \Z} \left( \bigoplus_{n \in \N} T^n(V) \right)_m$ quotiented by the ideal $(w \otimes w)$ with $w$ a homogeneous elements of $V$. It follows immediately that the objects of the above resolution are $\Z$-graded projective $A_n^e$-modules. 

We need now argue that the differentials are morphisms in the graded module category, which amounts to showing that they are degree-preserving. Consider $f,g \in A_n$ homogeneous of degrees $m_1,m_2$, respectively, and $v = v_1 \wedge \dots \wedge v_i \in \bigwedge^{i}V$ such that each $v_k$ is homogeneous of degree $l_k$, $1 \leq k \leq i$. Then $f \otimes (v_1 \wedge \dots \wedge v_i) \otimes g$ is of degree $m_1 + \sum_{j = 1}^{i}l_j + m_2 = m_1 + |v| + m_2$ in $A_n \otimes \bigwedge^{i}V \otimes A_n$. Now, considering the image of the differential above, notice that each term in the summand preserves the degree. Indeed, for any $1 \leq j \leq i$, we see
\begin{align*}
    (-1)^{j-1} \underbrace{(fv_j)}_{\deg m_1 + l_j} & \otimes \underbrace{(v_1 \wedge \cdots \hat{v}_j \cdots \wedge v_i)}_{\deg |v| - l_j} \otimes \underbrace{g}_{\deg m_2} \\
    + (-1)^j \underbrace{f}_{\deg m_1} & \otimes \underbrace{(v_1 \wedge \cdots \hat{v}_j \cdots \wedge v_i)}_{\deg |v| - l_j} \otimes \underbrace{(v_jg)}_{\deg |v| + m_2}
\end{align*}
has total degree $m_1 + |v| + m_2$ in $A_n \otimes \bigwedge^{i-1}V \otimes A_n$. Hence, the above resolution provides a projective resolution of $A_n$ in the category of $\Z$-graded $A^e$-modules. It immediately follows that $\Ext^{2n}_{A^e}(A,A^e) \cong A_n$ in this category. We conclude that $A_n$ is $\Z$-graded Calabi--Yau of Gorenstein index $0$ for all $n$. 

Let us now shift our view to see $A_n$ with a $\Z^n$-grading: set $|x_i| = \mathbf{e_i} \in \Z^n$ and $|y_i| = - \mathbf{e_i} \in \Z^n$. As above, we have an induced $\Z^n$-grading on the exterior algebra, and the differentials in the above resolution preserve degree just the same. In particular, we have that $A_n$ is $\Z^n$-graded Calabi--Yau of Gorenstein index $0$ for all $n$. In fact, the same observations show that for the same will hold true for any abelian group grading on $A_n$.
\end{example}

\begin{example}[Generalized Weyl Algebras]
First introduced by Bavula in \cite{bavula1992generalized} and later generalized in \cite{bavula1996tensor}, the generalized Weyl algebras (GWAs) of degree $n$ over a ring $D$ are defined as follows. Let $\sigma = (\sigma_1, \dots, \sigma_n)$ be a tuple of $n$ commuting automorphisms of $D$, and let $a = (a_1, \dots, a_n)$ be a tuple of $n$ central elements of $D$ such that $\sigma_i(a_j) = a_j$ for all $i \neq j$. Then the degree $n$ GWA over $D$ is the ring $A = D(\sigma, a)$ generated by $D$ and $2n$ indeterminates $x_1, \dots, x_n, y_1, \dots, y_n$ which are subject to the relations
\begin{align*}
    y_ix_i &= a_i, \\
    x_iy_i &= \sigma_i(a_i), \\
    x_i\alpha &= \sigma_i(\alpha) x_i \quad \forall \alpha \in D, \\
     y_i \alpha &= \sigma_i^{-1}(\alpha)y_i \quad \forall \alpha \in D, \\
    [x_i,x_j] &= [y_i,y_j] = [x_i,y_j] = 0 \quad \forall i \neq j.
\end{align*}
Any degree $n$ GWA further possesses a $\Z^n$-grading by setting $|x_i| = \mathbf{e_i}$ and $|y_i| = -\mathbf{e_i}$ (cf. \cite{bavula1992generalized}). 

As in \cite{Liu-HomSmoothGWAI}, consider $D = k[z]$, $\sigma = \sigma_1 = \lambda z + \eta$ for $\lambda \neq 0, \eta \in k$, and $a = a_1$ a polynomial in $k[z]$. Theorem 1.1 of \cite{Liu-HomSmoothGWAI} tells us that the resulting degree $1$ GWA $A = D(\sigma,a)$ is homologically smooth, and in particular twisted Calabi--Yau of dimension $2$, if and only if $a$ has no multiple roots. As in our above discussion, $A$ also has a $\Z$-grading, and is hence $\Z$-graded twisted Calabi--Yau of dimension $2$ by Theorem \ref{thm:gradcyiff}. Further, \cite[Theorem 1.1]{LIU2018228} gives conditions under which a GWA over $D = k[z_1,z_2]$ is twisted Calabi--Yau, which again by our previous discussions, yields examples of $\Z^2$-graded twisted Calabi--Yau algebras. It would be interesting to see more examples of Calabi--Yau algebras arising from GWAs over more general rings $D$. As noted above, all GWAs possess a $\Z^n$-grading, so this would in turn produce more examples of group-graded twisted Calabi--Yau algebras.
\end{example}

\begin{example}[Universal Enveloping Algebras of Finite-Dimensional Lie Algebras]
    Let $\g$ be a $G$-graded Lie algebra with vector space decomposition $\g = \bigoplus_{\sigma \in G} \g_\sigma$ that is compatible with the bracket operation; that is, $[\g_\sigma,\g_\tau] \subseteq \g_{\sigma \tau}$ for all $\sigma, \tau \in G$. It is well-known that if $G$ is an abelian group, then the universal enveloping algebra of $\g$, $U(\g)$, inherits this $G$-grading \cite{yasumura2023universal}, and that if $\g$ is finite dimensional, then $U(\g)$ is twisted Calabi--Yau \cite{rigiddualuea}. Again, in light of Theorem \ref{thm:gradcyiff}, we see that universal enveloping algebras of finitely dimensional $G$-graded Lie algebras are $G$-graded twisted Calabi--Yau. 
\end{example}

\section{Localizations of twisted Calabi--Yau algebras} \label{sec:localization}

It is well-known that $k[t]$ provides an example of a $1$-dimensional $\N$-graded Calabi--Yau algebra, and we can observe that localizing at $t$ turns this connected, $\N$-graded algebra into the connected, strongly $\Z$-graded algebra $k[t,t^{-1}]$. We will show that localization at left and right denominator subsets preserves the twisted Calabi--Yau property, which in turn will give us a new way of finding group-graded twisted Calabi--Yau algebras.

Indeed, closely related to the study of Calabi--Yau algebras is that of Artin--Schelter regular algebras, a definition of which can be found in \cite{rogalski2023artinschelter}, for example. For our purposes, we will take the definition of an Artin--Schelter regular algebra to be a connected $\N$-graded twisted Calabi--Yau algebra, as they were proven to be equivalent in \cite[Lemma 1.2]{Reyes_2014}.  As such, we will find a rich source of $\Z$-graded twisted Calabi--Yau algebras arising as localizations of Artin--Schelter regular algebras.

We start by stating two well-known technical lemmas, the proofs of which we omit. Recall that a multiplicative subset $S$ of a ring $R$ is said to be a left (resp. right) denominator subset if it is left (resp. right) permutable and left (resp. right) reversible. That is,
\begin{itemize}
    \item for all $a \in R$ and $s \in S$, $Sa \cap Rs \neq \emptyset$ (resp. $aS \cap sR \neq \emptyset$),
    \item for all $a \in R$, if $as' = 0$ for some $s' \in S$, then $sa = 0$ for some $s \in S$ (resp. if $s'a = 0$ for some $s' \in S$, then $as = 0$ for some $s \in S$).
\end{itemize}

\begin{lem} \label{leftrightdemon}
    Let $A$ and $B$ be $k$-algebras such that $S_1 \subseteq A$ is a left (resp. right) denominator subset, and such that $S_2 \subseteq B$ is also a left (resp. right) denominator subset. Then, 
    $$T := S_1 \otimes S_2 = \{(s_1 \otimes 1)(1 \otimes s_2) = (1 \otimes s_2)(s_1 \otimes 1) = (s_1 \otimes s_2): s_i \in S_i\}$$
    is a left (resp. right) denominator subset of $A \otimes B$.
\end{lem}

We are most interested in the case where we have $A$ a $k$-algebra viewed as a left $A^e$-module. If $S$ is a left (resp. right) denominator subset of $A$, then we can consider localization of $A$ as a left $A$-module, $S^{-1}A$ (resp. as a right $A$-module $AS^{-1}$). To make sense of the localization of $A$ viewed as a left $A^e$-module, we want to consider $T = S \otimes S^{\op}$ for $S \subseteq A$ a left and right denominator set, as in Lemma \ref{leftrightdemon}.

\begin{lem} \label{local@tensor}
    Let $A$ and $B$ be $k$-algebras. Let $S_1 \subseteq A$ be a left and right denominator subset, let $S_2 \subseteq B$ be a left and right denominator subset, and define $T \coloneqq S_1 \otimes S_2$ as in the Lemma \ref{leftrightdemon}. Then $T^{-1}(A \otimes B) \cong S_1^{-1}A \otimes S_2^{-1}B$ as $k$-algebras.  
\end{lem}

\begin{lem} \label{localsmooth}
    Let $A$ be a $k$-algebra. If $A$ is homologically smooth, then $S^{-1}A$ is homologically smooth, where $S$ is both a left and right denominator set of $A$. 
\end{lem}

\begin{proof}
    Let $A$ be as above. Since $A$ is homologically smooth, there exists a finite length projective resolution $P_\bullet \to A \to 0$ of finitely generated projective $A^e$-modules. Let $S \subseteq A$ be a left and right denominator subset. It is clear that $S^{\op} \subseteq A^{\op}$ is a left and right denominator subset from our observations above, so we may define the left and right denominator subset $\S := S \otimes S^{\op} \subseteq A \otimes A^{\op} = A^e$ as in Lemma \ref{leftrightdemon}. Let us now localize $P_\bullet \to A \to 0$ at $\S$. Since localization is exact, $\S^{-1}P_\bullet \to \S^{-1}A \to 0$ remains an exact sequence of $\S^{-1}A^e$-modules. It is well known that localizations of finitely generated projective modules remain finitely generated projective modules over the localized ring; hence, $\S^{-1}P_{\bullet} \to \S^{-1}A \to 0$ is a projective resolution of finitely generated $\S^{-1}A^e$-modules.
    
    Letting $B = A^{\op}$ and $S_2 = S^{\op}$ in Lemma \ref{local@tensor}, we know that $\S^{-1}A^e \cong S^{-1}A \otimes (S^{\op})^{-1}A^{\op}$ as rings. Since $S \subseteq A$ is a left and right denominator set of $A$, we also have
    \begin{align*}
        S^{-1}A \otimes (S^{\op})^{-1}A^{\op} &\cong S^{-1}A \otimes (AS^{-1})^{\op} \\
        &\cong S^{-1}A \otimes (S^{-1}A)^{\op} \\
        &= (S^{-1}A)^e.
    \end{align*}

    In particular, this allows us to view the $\S^{-1}A^e$-module resolution $\S^{-1}P_\bullet \to \S^{-1}A \to 0$ instead as that of $(S^{-1}A)^e$-modules. Further,
    \begin{align*}
        \S^{-1}A &\cong \S^{-1}A^e \otimes_{A^e} A \\
        &\cong \left( S^{-1}A \otimes (S^{-1}A)^{\op} \right) \otimes_{A^e} A \\
        &\cong S^{-1}A \otimes_A A \otimes_A S^{-1}A \\
        &\cong S^{-1}A \otimes_A S^{-1}A \\
        &\cong S^{-1}(S^{-1}A) \\
        &\cong S^{-1}A
    \end{align*}
    as $(S^{-1}A)^e$-modules. Thus, $\S^{-1}P_\bullet \to S^{-1}A \to 0$ provides a finite length resolution of finite generated and projective $(S^{-1}A)^e$-modules. We conclude that $S^{-1}A$ is homologically smooth, as desired.
\end{proof}

\begin{lem} 
\label{localhh} [Localization and Hochschild (Co)homology]
Let $A$ be a $k$-algebra which is finitely presented as a left $A^{e}$-module, and let $M$ be an $A^e$-module. Suppose further that $S \subseteq A$ is a left and right denominator set so that we may define $\S  := S \otimes S^{\op} \subseteq A^e$ as in Lemma \ref{localsmooth}. Then, the following hold as isomorphisms of $\S^{-1}A^e$-modules:
\begin{align*}
  \S^{-1}\text{HH}^{i}(A,M) &\coloneqq \S^{-1}\ext_{A^e}^{i}(A,M) \\
&\hphantom:\cong \\
\ext^i_{\S^{-1}A^e}(\S^{-1}A, \S^{-1}M) &\hphantom:=: \text{HH}^{i}(S^{-1}A, \S^{-1}M).
\end{align*}
\end{lem}

\begin{proof}
    Let $P_\bullet \to A \to 0$ a projective $A^e$-module resolution of $A$ and consider 
    \begin{align*}
        C^\bullet &:= 0 \to \Hom_{A^e}(P_0, A^e) \to \Hom_{A^e}(P_1, A^e) \to \dots, \\
        \S^{-1} C^{\bullet} &= 0 \to \S^{-1}\Hom_{A^e}(P_0, A^e) \to \S^{-1}\Hom_{A^e}(P_1, A^e) \to \dots.
    \end{align*}
    Since the localization functor taking $\mymod{A^e}$ to $\mymod{\S^{-1}A^e}$ is additive exact, we have
    \begin{align*}
        H^i(\S^{-1}C^\bullet) &\cong \S^{-1} H^i(C^{\bullet}).
    \end{align*}
    Now, it is well-known that because $\S^{-1}A^e$ is flat over $A^e$ and $A$ is finitely presented as an $A^e$-module, then 
    \begin{align*}
        \S^{-1}A^e \otimes_{A^e} \Hom_{A^e}(A,M) \cong \Hom_{\S^{-1}A^e}(\S^{-1}A^e \otimes_{A^e} A, \S^{-1}A^e \otimes_{A^e} M).
    \end{align*}
    Indeed, the same proof given in \cite[Proposition 2.10]{eisenbud} will work for our more general case. With this isomorphism in hand, we observe the following:
    \begin{align*}
        \S^{-1}\Hom_{A^e}(A,N) &\cong \Hom_{\S^{-1}A^e}(\S^{-1}A, \S^{-1}M) \\
        &\cong \Hom_{\S^{-1}A^e}(S^{-1}A, \S^{-1}M).
    \end{align*}
    Notice that the last isomorphism was obtained by recalling that $\S^{-1}A \cong S^{-1}A$ by Lemma \ref{localsmooth}. We conclude that localization preserves Hochschild (co)homology, as desired.
\end{proof}

\begin{lem} \label{homsmoothfp}
Let $A$ be a $k$-algebra which is homologically smooth. Then $A$ is finitely presented as a left $A^e$-module.
\end{lem}

\begin{proof}
Since $A$ is homologically smooth, there is a resolution $0 \to P_n \to \dots \to P_1 \xrightarrow{d_1} P_0 \xrightarrow{d_0} A \to 0$ of finitely generated and projective $A^e$-modules. In particular, we have the short exact sequence $0 \to \ker(d_0) \to P_0 \to A \to 0$. Notice that $P_0$ is finitely presented as a left $A^e$-module, and $\ker(d_0) = \im(d_1)$. Since the homomorphic image of a finitely generated module is finitely generated, we have that $\ker(d_0)$ is finitely generated. We conclude that $A$ is finitely presented as a left $A^e$-module, as desired.
\end{proof}

\begin{lem} \label{bimoduleloc}
Let $A$ be a $k$-algebra with $S \subseteq A$ a left and right denominator set. If $U$ is an invertible $(A,A)$-bimodule, then $S^{-1}U \cong US^{-1}$ as $(A,A)$-bimodules.
\end{lem}

\begin{proof}
Let $A, S,$ and $U$ as above and let $V$ be such that $U \otimes_A V \cong V \otimes_A U \cong A$, as $(A,A)$-bimodules. Then, as will be argued below,
\begin{align}
S^{-1}US^{-1} \cong S^{-1}U \otimes_A AS^{-1} &\cong \Hom_{S^{-1}A}(S^{-1}V, S^{-1}A) \otimes_A AS^{-1} \label{eq:fp}\\
&\cong \Hom_{S^{-1}A}(S^{-1}V, S^{-1}A \otimes_A AS^{-1}) \label{eq:lam20}\\
&\cong \Hom_{S^{-1}A}(S^{-1}V, S^{-1}A) \label{eq:fp2} \\
&\cong S^{-1}U \label{eq:fp3}.
\end{align}
We briefly argue why each of the above isomorphisms holds:

(\ref{eq:fp},\ref{eq:fp2},\ref{eq:fp3}) 
Since $U$ and $V$ are invertible $(A,A)$-bimodules, $U \cong \Hom_A(V,A)$ and $V \cong \Hom(U,A)$ hold as isomorphisms of $(A,A)$-bimodules. Moving $AS^{-1}$ inside the parentheses is justified by \cite[Proposition 2.10]{eisenbud}, the proof of which holds in the noncommutative setting. Recall that $V$ and $U$ are invertible bimodules, and are thus finitely generated and projective $A$ modules on the left and right. In particular, they are finitely presented as left and right $A$-modules.

(\ref{eq:lam20})Notice that $V$ is finitely generated and projective as a left $A$-modules, so $S^{-1}V$ remains so as a left $S^{-1}A$-module. Now apply \cite[Section 2, Exercise 20]{sacred3}.

An analogous argument can be used to show that $S^{-1}US^{-1} \cong US^{-1}$, allowing us to conclude, indeed, that $S^{-1}U \cong US^{-1}$, as desired.
\end{proof}

At last, we are sufficiently prepared to prove the main result of this section.

\begin{thm} \label{thm:localization}
    Let $A$ be a $k$-algebra, and let $S$ be a left and right denominator subset of $A$. If $A$ is twisted Calabi--Yau of dimension $d$, then $S^{-1}A$ is twisted Calabi--Yau of dimension $d$.
\end{thm}

\begin{proof}
    Since $A$ is homologically smooth, it follows from Lemma \ref{localsmooth} that $S^{-1}A$ is also homologically smooth. Further, by Lemma \ref{homsmoothfp}, homological smoothness gives finite presentation of $A$ as an $A^e$-module.
    Let $U = \ext_{A^e}^{d}(A,A^{e})$. From Lemma \ref{localhh} and the Calabi--Yau property of $A$, it follows that $\ext^{i}_{\S^{-1}A^e}(S^{-1}A, \S^{-1}A^e)$ is zero for $i \neq d$ and isomorphic to $\S^{-1}U$ otherwise. It suffices to show that $\S^{-1}U$ is an invertible $(S^{-1}A)^e$-module. Towards that end, let $V$ be such that $U \otimes_A V \cong V \otimes_A U \cong A$ as $(A,A)$-bimodules. Then, as $(S^{-1}A)^e$-modules,
    \begin{align*}
        \S^{-1}U \otimes_{S^{-1}A} \S^{-1}V &\cong  S^{-1}U \otimes_{S^{-1}A} S^{-1}V 
        \cong S^{-1}(U \otimes_A V)
        \cong S^{-1}A.
    \end{align*}
    Recall that the first isomorphism is Lemma \ref{bimoduleloc} applied to the same argument used in Lemma \ref{localsmooth} to argue that $\S^{-1}A \cong S^{-1}A$ as $(S^{-1}A)^e$-modules.
    
    An analogous argument gives $\S^{-1}V \otimes_{S^{-1}A} \S^{-1}U \cong S^{-1}A$. We conclude that $\S^{-1}A$ is an invertible $(S^{-1}A)^e$-module with inverse given by $\S^{-1}V$. It follows that $S^{-1}A$ is CY of dimension $d$, as desired.
\end{proof}

We are now equipped to discuss examples arising from the ideas considered at the beginning of this section. 

\begin{example}[Localizations of polynomial rings over a field]
    Let $A = k[t_1, \dots, t_n]$. Because $A$ is commutative, any multiplicative subset will provide a left and right denominator subset. As a consequence of Theorem \ref{thm:localization}, we have that $k[t,t^{-1}]$ is Calabi--Yau of dimension 1. Further, since $k[t,t^{-1}]$ is $\Z$-graded with $\deg(t) = 1$ and $\deg(t^{-1}) = -1$, Theorem \ref{thm:gradcyiff} gives that $k[t,t^{-1}]$ is $\Z$-graded Calabi--Yau of dimension 1. Furthermore, by Theorem \ref{thm:localization} the field of fractions $k(t)$ is Calabi--Yau of dimension 1. Observe that $k(t)$ no longer carries an induced grading, since the multiplicative set we localize at contains non-homogeneous elements. Alternatively, we can recall that a field can only have trivial $\Z$-grading. One can construct many similar examples by considering different localizations of the $n$-dimensional Calabi--Yau algebras $k[t_1, \dots, t_n]$.
    \end{example}
    \begin{table}[!h]
        \centering
        \begin{tabular}{*3c}
        \toprule
             $k[t_1, \dots, t_n]$ & $k[t_1,t_1^{-1}, \dots, t_n,t_n^{-1}]$ & $k(t_1, \dots, t_n)$  \\
             Connected $\N$-graded & Strongly $\Z$-graded & Ungraded \\
             \bottomrule
        \end{tabular}
        \bigskip
        
        \caption{Different graded flavors of commutative, $n$-dimensional Calabi--Yau algebras.}
        \label{tab:my_label}
    \end{table} 
Outside the commutative setting, the question of finding suitable sets at which to localize becomes more difficult. Here is one potential strategy. Recall that an element $x$ of a ring $R$ is said to be \textit{normal} if $xR = Rx$. In particular, the set of normal elements of a ring will generate a multiplicatively closed right and left permutable subset. If $S$ is the set of normal elements of a ring $R$, then so long as $S$ is left and right reversible, $S$ will be a left and right denominator subset of $R$. In the case where $R$ is a domain, this comes for free, and we find applications in the following example.  

\begin{example} [Localizations of Artin--Schelter regular algebras]
    By Theorem \ref{thm:localization}, we can expect to find a large number of $\Z$-graded Calabi--Yau algebras arising as localizations of Artin--Schelter regular algebras, provided we have a suitable left and right denominator subset. 

    All known examples of Artin--Schelter regular algebras are domains, and it is conjectured that this is always the case. In light of the above discussion, localizing any (known) Artin--Schelter regular algebra at its set of homogeneous normal elements is an example of a $\Z$-graded Calabi--Yau algebra. Another choice of a left and right denominator set is given by $\{1, f, f^2, \dots\}$ for $f$ a homogeneous normal element. Further, given that the algebra has finite GK-dimension, we could also localize at the set of all nonzero elements, thus giving rise to more examples of Calabi--Yau division algebras.
\end{example}

\printbibliography

\end{document}